\documentclass{birkjour}
\usepackage{mathrsfs}
\newcommand{\inv}{\iota}

\newtheorem{Thm}{Theorem}
\newtheorem{Lemma}[Thm]{Lemma}

\newtheorem{Cor}[Thm]{Corollary}

\theoremstyle{remark}
\newtheorem*{Problem}{Problem}

\newtheorem*{Example}{Examples}

\theoremstyle{definition}
\newtheorem*{Def}{Definition}
\title[Homogeneous involution and graded identities]{Homogeneous involution on graded division algebras and their polynomial identities}
\author{Felipe Yukihide Yasumura}
\address{Department of Mathematics, Instituto de Matem\'atica e Estat\'istica, Universidade de S\~ao Paulo, Brazil}
\email{fyyasumura@ime.usp.br}


\begin{document}
\begin{abstract}
In this short note, we describe the so-called homogeneous involution on finite-dimensional graded-division algebra over an algebraically closed field. We also compute their graded polynomial identities with involution. As pointed out by L.~Fonseca and T.~de Mello, a homogeneous involution naturally appears when dealing with graded polynomial identities and a compatible involution.
\end{abstract}
\maketitle

\section{Introduction}
The main purpose of this short note is to investigate the homogeneous involution on finite-dimensional graded-division algebras over an algebraically closed field of characteristic zero and their graded polynomial identities with involution.

It is known that graded algebras play a fundamental role in several branches of Mathematics, being the main topic of research or being an important tool to understand an object. The monograph \cite{EK2013} compiles the state-of-art of the theory, where understanding the gradings is the central objective. It is worth mentioning that the classification of involutions that are compatible with a given grading, the so called \emph{degree-preserving involution} or \emph{graded involution}, is crucial in the approach.

On the other hand, an involution that inverts the degree appears naturally on several contexts, for intance, the Leavitt path algebras (see \cite{Hazrat}) and Matrix algebras endowed with an elementary grading and the transpose involution (see \cite{Mello1} for details, and see also \cite{FSY}).

Considering group gradings on the upper triangular matrices (a classification is obtained in \cite{VinKoVa2004,VZ2007}), one sees that the natural transpose with respect to the secondary diagonal sends a homogeneous element to a homogeneous element. If the grading is adequate, the natural involution will map a whole homogeneous component in another homogeneous component. So, T.~de Mello studied and described the so called \emph{homogeneous involution} on upper triangular matrix algebra \cite{Mello}.

Following the paper by T.~de Mello, we shall investigate the homogenous involutions on finite-dimensional graded division algebras over an algebraically closed field.

Next, it is known the relevance of the description of the T-ideal of all polynomial identities satisfied by a given algebra, and the related main problems of the theory. We shall describe the graded polynomial identities with homogeneous involution on the graded division algebras endowed with a homogeneous involution. The graded version is done in \cite{AljHN}, and in \cite{HN} the graded identities with an involution were computed for certain types of gradings.

Finally, it is worth mentioning the following phenomena (see \cite{Mello2}) when dealing with the free graded algebra with a homogeneous involution (we shall give the precise construction in the next section). Let $X_G=\bigcup_{g\in G}X_g$, where $X_g=\{x_1^{(g)},x_2^{(g)},\ldots\}$, for each $g\in G$. Let $\mathbb{F}\langle X,\inv\rangle$ be the free $G$-graded algebra with an involution. Then, since the unary operation $\inv$ is an antiautomorphism, one obtains $\inv(x^{(g)}x^{(h)})=\inv(x^{(h)})\inv(x^{(g)})$. If we require that $\inv$ preserves the homogeneous degree, then this equation makes sense only if the grading group is abelian. Otherwise, there should be a antihomomorphism $\tau:G\to G$ such that $\deg_G\inv(x^{(g)})=\tau(g)$. Thus, in general, it is seems to be natural to consider homogeneous involutions in the context of graded polynomial identities endowed with a compatible involution.

\section{Preliminaries}
\subsection{Graded algebra} Let $G$ be any group. We use the multiplicative notation for $G$, and denote its neutral element by $1$. We say that an algebra $\mathcal{A}$ is $G$-graded if there exists a vector-space decomposition $\mathcal{A}=\bigoplus_{g\in G}\mathcal{A}_g$ such that $\mathcal{A}_g\mathcal{A}_h\subseteq\mathcal{A}_{gh}$, for all $g,h\in G$. The choice of the decomposition is called a \emph{$G$-grading}, and one usually denotes by $\Gamma$. The subspace $\mathcal{A}_g$ is called \emph{homogeneous component of degree $g$}. A nonzero element $x\in\mathcal{A}_g$ is called a homogeneous element of degree $g$. We denote $\deg_Gx=g$. The \emph{support} of the grading $\Gamma$ is $\mathrm{Supp}\,\Gamma=\{g\in G\mid\mathcal{A}_g\ne0\}$. By abuse of language, we shall denote the support by $\mathrm{Supp}\,\mathcal{A}$. A \emph{graded division algebra} is an associative algebra $\mathcal{D}$ with $1$, where each nonzero homogeneous element $x\in\mathcal{D}$ is invertible. 

Finally, we provide a precise definition of the following:
\begin{Def}
Let $\mathcal{A}=\bigoplus_{g\in G}\mathcal{A}_g$ be a $G$-graded algebra, and let $\tau:G\to G$ be a map. An involution $\psi$ on $\mathcal{A}$ is a \emph{homogeneous involution} with respect to $\tau$ or a $\tau$-involution if $\psi(\mathcal{A}_g)\subseteq\mathcal{A}_{\tau(g)}$, for all $g\in G$.
\end{Def}
In this paper, \emph{involution} will mean an $\mathbb{F}$-linear map involution. Also, we are specially interested in the case where the map $\tau$ is an anti-automorphism of order $2$ of the grading group.

\begin{Example}\ \\
\begin{enumerate}
\item If $G$ is an abelian group, then every degree-preserving involution is a homogeneous involution with respect to the identity map of $G$.
\item A {degree-inverting involution} is a homogeneous involution with respect to the inversion of $G$. It is worth mentioning that the degree-inverting involution on matrix algebras and upper triangular matrices were described in \cite{Mello1,FSY}.
\item If $\mathcal{D}=\bigoplus_{g\in G}\mathcal{D}_g$ is a graded-division algebra (where $G=\mathrm{Supp}\,\mathcal{D}$) and $\inv$ is a $\tau$-involution on $\mathcal{D}$, then $\tau$ is an involution of $G$. Indeed, for any $g$, $h\in G$, let $x_g$ and $x_h$ be nonzero homogeneous elements of $G$-degrees $g$ and $h$, respectively. Then $xy\ne0$ and
$$
\tau(gh)=\deg_G\inv(x_gx_h)=\deg_G(\inv(x_h)\inv(x_g))=\deg_G\inv(x_h)\deg_G\inv(x_g)=\tau(h)\tau(g),
$$
$$
g=\deg_Gx_g=\deg_g\inv(\inv(x_g))=\tau(\tau(g)).
$$
Thus, $\tau$ is an anti-automorphism of order $2$.
\end{enumerate}
\end{Example}

\subsection{Factor sets} We let $T$ be a finite group, and $\mathbb{F}^\times$ denote the set of invertible elements of $\mathbb{F}$.

A map $\sigma:T\times T\to\mathbb{F}^\times$, is called a \emph{2-cocycle} or a \emph{factor set} if
$$
\sigma(u,v)\sigma(uv,w)=\sigma(u,vw)\sigma(v,w),\quad\forall u,v,w\in T.
$$
We denote the set of all 2-cocycles by $Z^2(T,\mathbb{F}^\times)$. Note that, using the usual multiplication, $Z^2(T,\mathbb{F}^\times)$ is an abelian group.

For each map $\mu:T\to\mathbb{F}^\times$, we define $\delta\mu:T\times T\to\mathbb{F}^\times$ by
$$
\delta\mu(u,v)=\mu(u)\mu(v)\mu(uv)^{-1},\quad u,v\in T.
$$
We define $B^2(T,\mathbb{F}^\times)=\{\delta\mu\mid\mu:T\to\mathbb{F}^\times\}$. An easy exercise shows that $B^2(T,\mathbb{F}^\times)$ is a subgroup of $Z^2(T,\mathbb{F}^\times)$. The 2nd cohomological group of $T$ is the quotient $H^2(T,\mathbb{F}^\times)=Z^2(T,\mathbb{F}^\times)/B^2(T,\mathbb{F}^\times)$.

We can construct algebras from factor sets. Given an arbitrary map $\sigma:T\times T\to\mathbb{F}^\times$ denote by $\mathbb{F}^\sigma T$ the following algebra: $\mathbb{F}^\sigma T$ has a basis $\{X_u\mid u\in T\}$, and the product is defined by $X_uX_v=\sigma(u,v)X_{uv}$. Note that $\mathbb{F}^\sigma T$ is associative if and only if $\sigma\in Z^2(T,\mathbb{F}^\times)$. For instance, if $\sigma=1$ (the constant function), then $\mathbb{F}^\sigma T$ is the group algebra of $T$. Clearly such algebra have a natural $T$-grading. It is known that $\mathbb{F}^{\sigma}T\cong\mathbb{F}^{\sigma'}T$ if and only if $[\sigma]=[\sigma']$ (equality in $H^2(T,\mathbb{F}^\times)$.

\subsection{Graded division algebra}
 Assume that $\mathbb{F}$ is algebraically closed and let $\mathcal{D}=\bigoplus_{g\in G}\mathcal{D}_g$ be a finite-dimensional graded division algebra over $\mathbb{F}$. Let $T=\{g\in G\mid\mathcal{D}_g\ne0\}$ be its support. Then it is easy to see that $T$ is a subgroup of $G$. We use multiplicative notation for the product of $T$, and denote by $1$ its neutral element.

Moreover, $\mathcal{D}_1\supseteq\mathbb{F}$ is a division algebra. So $\mathcal{D}_1=\mathbb{F}$, since $\mathbb{F}$ is algebraically closed and $\dim_\mathbb{F}\mathcal{D}_1<\infty$. This also implies $\dim\mathcal{D}_g=1$, for all $g\in T$. Let $\{X_u\mid u\in T\}$ be a homogeneous basis of $\mathcal{D}$. Then $X_uX_v=\sigma(u,v)X_{uv}$, for some $\sigma(u,v)\in\mathbb{F}^\times$. Since $\mathcal{D}$ is associative, from $(X_uX_v)X_w=X_u(X_vX_w)$, we derive that $\sigma$ is a 2-cocycle. Hence, $\mathcal{D}\cong\mathbb{F}^\sigma T$, the twisted group algebra of $T$ by $\sigma$. Conversely, for any finite group $T$ and any $\sigma\in Z^2(T,\mathbb{F}^\times)$, the natural $T$-grading on $\mathbb{F}^\sigma T$ turns it into a graded division algebra.

\subsection{Free graded algebra with homogeneous involution}
We shall provide a construction of the free algebra endowed with a homogeneous involution. This is done using a particular case of the (relatively) free universal algebra in an adequate variety (see, for instance, \cite{Razmyslov} for a general discussion, and \cite{BY,BY2} as well for a particular graded version). Let $G$ be any group, and $X^G=\bigcup_{g\in G}X^{(g)}$, where $X^{(g)}=\{x_1^{(g)},x_2^{(g)},\ldots\}$. Let $\tau:G\to G$ be an involution, that is, an anti-automorphism of order $2$. We define the \emph{free $G$-graded associative algebra with a homogeneous involution} with respect to $\tau$, $\mathbb{F}\langle X^G,\inv\rangle$, in the following way. First, we let $\mathbb{F}\{X^G,\inv\}$ denote the absolute free $G$-graded binary algebra endowed with an unary operation (denote by $\inv$). We define the following polynomial identities
\begin{align}\label{eqvariety}
\begin{split}
x_1^{(g_1)}(x_2^{(g_2)}x_3^{(g_3)})-(x_1^{(g_1)}x_2^{(g_2)})x_3^{(g_3)}&=0\\
\inv(x_1^{(g_1)}x_2^{(g_2)})-\inv(x_2^{(g_2)})\inv(x_1^{(g_1)})&=0\\
\inv(\inv(x^{(g)}))-x^{(g)}&=0\\
\deg_G\inv(x^{(g)})=\tau(\deg_Gx^{(g)}).
\end{split}
\end{align}
The first relation is the associativity while the second and third indicate that $\inv$ acts as an involution. The last relation is also a polynomial identity, but to see this we need to describe the $G$-grading in terms of the projections (see, for instance, \cite{BY}). For each $g\in G$, let $\pi_g$ denote the unary operation on a $G$-graded algebra $\mathcal{A}$ given by the projection and inclusion $\pi_g:\mathcal{A}\to\mathcal{A}$. Then, the absolute free $G$-graded algebra $\mathbb{F}\{X^G,\inv\}$ is a quotient of the absolute free $\Omega$-algebra, where $\Omega$ contains one binary operation and $|G|+1$ unary operations (corresponding to each projection, and the involution). The quotient is given by the relations that define the $G$-grading, that is, $\pi_g(\pi_h(x))=\delta_{gh}\pi_h(x)$ and $\pi_g(\pi_{g_1}(x)\pi_{g_2}(y))=\delta_{g,g_1g_2}\pi_{g}(xy)$. Hence, in the language of this $\Omega$-algebra, the last equation of \eqref{eqvariety} is equivalent to
$$
\inv(\pi_g(x))-\pi_{\tau(g)}(\inv(x))=0.
$$
Using either the absolute free $G$-graded algebra or the (relatively) free $\Omega$-algebra, the free $G$-graded algebra with a $\tau$-involution $\mathbb{F}\langle X^G,\inv\rangle$ is the quotient of $\mathbb{F}\{X^G,\inv\}$ by the identities \eqref{eqvariety}. As mentioned before, the forth identity is natural in the context of graded polynomial identities with an involution.

As discussed in \cite{Mello2}, in the special case where $\inv$ is a degree-preserving involution, then we can define the new variables $x_+^{(g)}:=x^{(g)}+\inv(x^{(g)})$ and $x_-^{(g)}=x^{(g)}-\inv(x^{(g)})$ (the symmetric and skew symmetric variables). Then we get the classical construction of the free (graded) $\ast$-algebra. Since $\inv$ does not necessarily preserve the homogeneous degree, we cannot use such technique, since these variables would not be homogeneous.

Given a $G$-graded algebra $\mathcal{A}$ with a homogeneous involution $\inv$, we denote by $\mathrm{Id}_G(\mathcal{A})$ its ideal of graded polynomial identities, and by $\mathrm{Id}_{G,\inv}(\mathcal{A})$ the set of all of its graded polynomial identities with involution $\inv$.

\subsection{Codimension sequence}
Let $T$ be a finite group and consider the natural $T$-grading on $\mathbb{F}^\sigma T$. Given a sequence $s=(u_1,\ldots,u_m)\in T^m$, we let
$$
P^s_m=\mathrm{Span}\left\{x_{\pi(1)}^{(u_{\pi(1)})}\cdots x_{\pi(m)}^{(u_{\pi(m)})}\mid\pi\in\mathcal{S}_m\right\},\quad P^s_m(\mathbb{F}^\sigma T)=P^s_m/P^s_m\cap\mathrm{Id}_T(\mathbb{F}^\sigma T),
$$
where $\mathcal{S}_m$ is the symmetric group on the set of $m$ elements. The \emph{graded codimension sequence} of $\mathbb{F}^\sigma T$ is
$$
c^T_m(\mathbb{F}^\sigma T)=\dim\sum_{s\in T^m}P^s_m(\mathbb{F}^\sigma T).
$$
Now, given $I=(j_1,\ldots,j_m)\in\{0,1\}^m$, let
$$
P_{s,I}^T=\mathrm{Span}\left\{\inv^{j_1}\left(x_{\pi(1)}^{(u_{\pi(1)})}\right)\cdots\inv^{j_m}\left(x_{\pi(m)}^{(u_{\pi(m)})}\right)\mid\pi\in\mathcal{S}_m\right\}.
$$
As before, we let
$$
P_{s,I}^T(\mathbb{F}^\sigma T)=P_{s,I}^T/P_{s,I}^T\cap\mathrm{Id}_{T,\inv}(\mathbb{F}^\sigma T),\quad c_m^{T,\inv}(\mathbb{F}^\sigma T)=\dim\sum_{s,I}P_{s,I}^T(\mathbb{F}^\sigma T).
$$
The graded exponent and the graded-involution exponent (if exists) are respectively defined by
$$
\mathrm{exp}^T(\mathbb{F}^\sigma T)=\lim\limits_{m\to\infty}\sqrt[m]{c_m^T(\mathbb{F}^\sigma T)},\quad\mathrm{exp}^{T,\inv}(\mathbb{F}^\sigma T)=\lim\limits_{m\to\infty}\sqrt[m]{c_m^{T,\inv}(\mathbb{F}^\sigma T)}
$$

\section{Homogeneous involution on graded division algebra}
We let $\mathbb{F}$ be a field of characteristic not 2. Let $T$ be a finite group. Given $\sigma\in Z^2(T,\mathbb{F}^\ast)$, we consider the natural $T$-grading on $\mathbb{F}^\sigma T$.


We shall denote by $\mathrm{Aut}(T)$ the group of all automorphisms of $T$, and by $\overline{\mathrm{Aut}}(T)$ the group of all automorphism and antiautomorphism of $T$. For each $\varphi\in\mathrm{Aut}(T)$, let $\varphi\sigma:T\times T\to\mathbb{F}^\ast$ be defined by
$$
\varphi\sigma(u,v)=\sigma(\varphi(u),\varphi(v)).
$$
On the other hand, if $\psi\in\overline{\mathrm{Aut}}(T)$ is an antiautomorphism, then let $\psi\sigma:T\times T\to\mathbb{F}^\times$ be defined by
$$
\psi\sigma(u,v)=\sigma(\psi(v),\psi(u)).
$$
The next lemma is an elementary exercise. We include the proof for completeness.
\begin{Lemma}
For each $\theta\in\overline{\mathrm{Aut}}(T)$ and $\sigma'\in Z^2(T,\mathbb{F}^\times)$, $\theta\sigma'\in Z^2(T,\mathbb{F}^\times)$.
\end{Lemma}
\begin{proof}
Assume that $\theta$ is an antiautomorphism. Then, for any $u$, $v$, $w\in T$, we have
\begin{align*}
\theta\sigma'(u,v)\theta\sigma'(uv,w)&=\sigma'(\theta(v),\theta(u))\sigma'(\theta(w),\theta(uv))\\%
&=\sigma'(\theta(w),\theta(v))\sigma'(\theta(w)\theta(v),\theta(u))\\%
&=\theta\sigma'(v,w)\theta\sigma'(u,vw).
\end{align*}
In an analogous way we show that $\theta\sigma'\in Z^2(T,\mathbb{F}^\times)$ if $\theta$ is an automorphism of $T$.
\end{proof}
Hence, it is easy to see that we have an action of $\overline{\mathrm{Aut}}(T)$ on $Z^2(T,\mathbb{F}^\times)$. This action factors through $B^2(T,\mathbb{F}^\times)$:
\begin{Lemma}
If $\sigma'\in B^2(T,\mathbb{F}^\times)$ and $\theta\in\overline{\mathrm{Aut}}(T)$, then $\theta\sigma'\in B^2(T,\mathbb{F}^\times)$.
\end{Lemma}
\begin{proof}
Let $\sigma'=\delta\mu$. Then $\theta\delta\mu=\delta(\mu\circ\theta)\in B^2(T,\mathbb{F}^\times)$.
\end{proof}

Thus, we get:
\begin{Cor}
There is an action of $\overline{\mathrm{Aut}}(T)$ on $H^2(T,\mathbb{F}^\times)$.\qed
\end{Cor}

Now, the next result states the conditions so that we do have a homogeneous involution on $\mathbb{F}^\sigma T$. For this, we need maps satisfying the following condition:
\begin{Def}\label{conj}
Let $\sigma\in Z^2(T,\mathbb{F}^\times)$ and $\tau\in\overline{\mathrm{Aut}}(T)$. We say that the pair $(\sigma,\tau)$ is \emph{compatible} if there is a map $\mu:T\to\mathbb{F}^\times$ such that $\sigma=\delta\mu\cdot\tau\sigma$ (where $\cdot$ is the product of $Z^2(T,\mathbb{F}^\times)$) and $\mu(u\tau(u))=1$, for all $u\in T$.
\end{Def}
For instance, let $\tau:u\in T\mapsto u^{-1}\in T$ be the inversion. Then, for any $\sigma\in Z^2(T,\mathbb{F}^\times)$ such that $[\sigma]^2=1$, the pair $(\sigma,\tau)$ is compatible (see the proof of \cite[Proposition 4]{FSY}).

\begin{Thm}
Let $\tau:T\to T$. Then there exists a $\tau$-homogeneous involution on $\mathbb{F}^\sigma T$ if and only if $\tau$ is an antiautomorphism of order $2$ and $(\sigma,\tau)$ is compatible.
\end{Thm}
\begin{proof}
 Let $\{X_u\mid u\in T\}$ be a homogeneous basis of $\mathbb{F}^\sigma T$. First, assume that $\inv$ is a $\tau$-homogeneous involution on $\mathbb{F}^\sigma T$. Let $\tau:T\to T$ and $\mu:T\to\mathbb{F}^\ast$ be the maps such that
$$
\inv(X_u)=\mu(u)X_{\tau(u)},\quad u\in T.
$$
Since $X_u=\inv(\inv X_u)=\mu(\tau(u))\mu(u)X_{\tau\tau(u)}$, one obtains $\tau^2=1$. It also implies that $\tau$ is a bijection. Since
\begin{equation}
\begin{split}\label{eq1}
\mu(uv)\sigma(u,v)X_{\tau(uv)}=\inv(X_uX_v)=\inv(X_v)\inv(X_u)\\=\sigma(\tau(v),\tau(u))\mu(v)\mu(u)X_{\tau(v)\tau(u)},
\end{split}
\end{equation}
we obtain $\tau(uv)=\tau(v)\tau(u)$. Hence, $\tau$ is an antiautomorphism. Note that \eqref{eq1} also shows that $\sigma=\delta\mu\cdot(\tau\sigma)$, thus $[\sigma]=[\tau\sigma]$. Finally, from $\inv(X_uX_{\tau(u)})=\inv(X_{\tau(u)})\inv(X_u)$, we get
$$
\mu(u\tau(u))\sigma(u,\tau(u))=\mu(u)\mu(\tau(u))\sigma(u,\tau(u)),
$$
thus $\mu(u\tau(u))=\mu(u)\mu(\tau(u))$. Since $X_u=\inv(\inv(X_u))=\mu(\tau(u))\mu(u)X_u$, we get that $\mu(u\tau(u))=1$, for all $u\in T$.

On the other hand, assume that $\tau$ is an antiautomorphism of order $2$ such that $(\sigma,\tau)$ is compatible. So let $\mu:G\to\mathbb{F}^\times$ satisfy $\sigma=\delta\mu\cdot\tau\sigma$. We claim that
$$
\inv(X_u):=\mu(u)X_{\tau(u)}
$$
is a $\tau$-homogeneous involution on $\mathbb{F}^\sigma T$. Indeed,
\begin{align*}
\inv(X_uX_v)&=\sigma(u,v)\mu(uv)X_{\tau(uv)}\\%
&=\sigma(u,v)\mu(uv)\sigma(\tau(v),\tau(u))^{-1}X_{\tau(v)}X_{\tau(u)}\\%
&=\mu(u)\mu(v)X_{\tau(v)}X_{\tau(u)}=\inv(X_v)\inv(X_u).
\end{align*}
Now,
$$
\inv(\inv(X_u))=\mu(u)\mu(\tau(u))X_{\tau^2(u)}=\mu(u)\mu(\tau(u))X_u.
$$
It remains to prove that $\mu(u)\mu(\tau(u))=1$. Since
$$
\mu(u)\mu(\tau(u))\mu(u\tau(u))^{-1}=\sigma(u,\tau(u))\tau\sigma(u,\tau(u))^{-1}=1,
$$
we obtain that $\mu(u)\mu(\tau(u))=\mu(u\tau(u))=1$, since $(\sigma,\tau)$ is compatible.
\end{proof}

\begin{Problem}
Classify the compatible pairs $(\sigma,\tau)$.
\end{Problem}

\section{Polynomial identities of graded division algebra}
In this section, we describe the graded polynomial identities with involution of $\mathbb{F}^\sigma T$. Let $\mathbb{F}$ be a field of characteristic zero and $\inv$ a $\tau$-homogeneous involution on $\mathbb{F}^\sigma T$. Let $U=(u_1, \ldots, u_m)\in T^m$, $I=(i_1,\ldots,i_m),I'=(i_1',\ldots,i_m')\in\{0,1\}^m$ and $\theta,\theta'\in\mathcal{S}_m$ be such that
\begin{equation}\label{eq2}
\tau^{i_{\theta'(1)}'}(u_{\theta'(1)})\cdots\tau^{i_{\theta'(m)}'}(u_{\theta'(m)})=\tau^{i_{\theta(1)}}(u_{\theta(1)})\cdots\tau^{i_{\theta(m)}}(u_{\theta(m)})
\end{equation}
Then we can find a constant $\alpha_{U,I,I',\theta,\theta'}\in\mathbb{F}$ such that
$$
\inv^{i'_{\theta'(1)}}(X_{u_{\theta'(1)}})\cdots\inv^{i'_{\theta'(m)}}(X_{u_{\theta'(m)}})=\alpha_{U,I,I',\theta,\theta'}\inv^{i_{\theta(1)}}(X_{u_{\theta(1)}})\cdots\inv^{i_{\theta(m)}}(X_{u_{\theta(m)}}).
$$
Hence, denoting $z_i=x_i^{(u_i)}$,
$$
f_{U,I,I',\theta,\theta'}=\inv^{i'_{\theta'(1)}}(z_{u_{\theta'(1)}})\cdots\inv^{i'_{\theta'(m)}}(z_{u_{\theta'(m)}})-\alpha_{U,I,\theta}\inv^{i_1}(z_{\theta(1)})\cdots\inv^{i_m}(z_{\theta(m)})
$$
is a $G$-graded polynomial identity with involution of $\mathbb{F}^\sigma T$, which shall be called an \emph{elementary} identity (following the graded case of \cite{AljHN}). Let $\mathscr{T}_U$ be the set of all triples $(U,I,I',\theta,\theta')$ such that \eqref{eq2} holds valid.

\begin{Thm}\label{mainthmPI}
$\mathrm{Id}_{T,\inv}(\mathbb{F}^\sigma T)$ is generated by $\{f_{U,I,I',\theta,\theta'}\mid(I,I',\theta,\theta')\in\mathscr{T}_U,\,|U|\le|T|\}$.
\end{Thm}
\begin{proof}
First, we shall prove that any multilinear polynomial identity is a consequence of the elementary ones. Then, we show that the elementary identities follow from the ones having length at most $|T|$. Let $\mathcal{I}$ be the $T_{T,\inv}$-ideal generated by $\{f_{U,I,I',\theta,\theta'}\mid(I,I',\theta,\theta')\in\mathscr{T}_U,\,|U|\le|T|\}$, and $\mathcal{J}$ be the $T_{T,\inv}$-ideal generated by $\{f_{U,I,I',\theta,\theta'}\mid(I,I',\theta,\theta')\in\mathscr{T}_U\}$.

Let $f\in\mathrm{Id}_{T,\inv}(\mathbb{F}^\sigma T)$ be a $G$-homogeneous polynomial. Since $\mathrm{char}\,\mathbb{F}=0$, we may assume that $f$ is multilinear. Write
\begin{equation}\label{summ}
f=\sum_{\substack{\theta\in\mathcal{S}_m\\I=(i_1,\ldots,i_m)}}\alpha_{\theta,I}p_{\theta,I},
\end{equation}
where $p_{\theta,I}=\inv^{i_{\sigma(1)}}(z_{\sigma(1)})\cdots\inv^{i_{\sigma(m)}}(z_{\sigma(m)})$. Since every monomial in \eqref{summ} satisfies $\tau^{i_{\sigma(1)}}(u_{\sigma(1)})\cdots\tau^{i_{\sigma(m)}}(u_{\sigma(m)})=\deg_Tf$, we see that $(I,I',\theta,\theta')\in\mathscr{T}_U$ for every pair of $(I,\theta)$ and $(I',\theta')$ appearing with nonzero coefficient in \eqref{summ}. Here, $U=(\deg_Tz_1,\ldots,\deg_Tz_m)$. Hence, modulo $\mathcal{J}$, $f$ is equal to a single monomial, up to a scalar. As $\mathbb{F}^\sigma T$ is a graded division algebra, a monomial cannot be a graded polynomial identity of it. Thus, $f=0$ modulo $\mathcal{J}$, so $f\in\mathcal{J}$.

Finally, let $p=f_{U,I,I',\theta,\theta'}$ with $|U|>|T|$. We shall prove by induction on $|U|$ that $f\in\mathcal{I}$. For the sake o simplicity, we may replace $\theta$ by the identity, and then $\theta'$ becomes $\theta''=\theta'\theta^{-1}$, and we may assume that $I=(0,0,\ldots,0)$ and $I'$ is replaced by $I''=(i_1'+i_1,\ldots,i_m'+i_m)$ (where the sum is taken modulo $2$). Formally, we shall call $z_j=\inv^{i_j}(x_{\theta(j)}^{u_{\theta(j)}})$, and then
$$
p=z_1\cdots z_m-\inv^{i_1''}(z_{\theta''(1)})\cdots\inv^{i_m''}(z_{\theta''(m)}).
$$
To make notations cleaner, we may also suppress the double prime, so we shall write $p=f_{U,(0,\ldots,0),I,1,\theta}$. Denote also $v_i=\tau(u_{\theta(i)})$.

Consider the elements $\{v_1,v_1v_2,\ldots,v_1v_2\cdots v_m\}$, having exactly $|U|>|T|$ elements. By the pigeonhole principle, there should be two distinct sequences whose product coincides. Thus, there exists $1<i<j\le m$ such that $v_iv_{i+1}\cdots v_j=1$. Since $(v_i\cdots v_{j-1})v_j=v_j(v_i\cdots v_{j-1})=1$, we may, modulo $\mathcal{I}$, cyclically permute the variables $\inv^{i_i}(z_i)\cdots\inv^{i_j}(z_j)$. As $(v_i\cdots v_j)v_k=v_k(v_i\cdots v_j)$ for any $k$, we may also move the string $\inv^{i_i}(z_i)\cdots\inv^{i_j}(z_j)$ anywhere in the last monomial, modulo $\mathcal{I}$. Finally, since $\tau(1)=1$, we may apply $\inv$ to $\inv^{i_i}(z_i)\cdots\inv^{i_j}(z_j)$, modulo $\mathcal{I}$.

Hence, modulo $\mathcal{I}$, we may assume that there exists $\ell$ such that $z_\ell$ and $z_{\ell+1}$ appears consecutively (in this order) in the second monomial of $p$, and $\inv$ is applied in the product or $z_\ell z_{\ell+1}$ or not. More precisely, modulo $\mathcal{I}$, either
$$
p=w_1z_{\ell}z_{\ell+1}w_2-w_1'z_{\ell}z_{\ell+1}w_2'
$$
or
$$
p=w_1z_{\ell}z_{\ell+1}w_2-w_1'\inv(z_{\ell}+1)\inv(z_{\ell})w_2'.
$$
Defining $g=u_\ell u_{\ell+1}$ we have that $p$ is a consequence of either one of the elementary identities $w_1x^{(g)}w_2-w_1'x^{(g)}w_2'$ or $w_1x^{(g)}w_2-w_1'\inv\left(x^{(g)}\right)w_2'$. In any case, $p$ is a consequence of an elementary identity of total degree $|U|-1$. By induction, this implies that $p\in\mathcal{I}$, proving the result.
\end{proof}

Finally, we shall obtain some estimates to the graded-involution codimension sequence of the $\mathbb{F}^\sigma T$. We shall prove the following:
\begin{Thm}\label{mainthmcodim}
Let $T$ be a finite group, and $\sigma\in Z^2(T,\mathbb{F}^\times)$. Assume that $\mathbb{F}^\sigma T$ admits a $\tau$-homogeneous involution $\inv$. Then
\begin{enumerate}
\item $c_m^T(\mathbb{F}^\sigma T)\le c_m^{T,\inv}(\mathbb{F}^\sigma T)\le|T|c_m^T(\mathbb{F}^\sigma T)$, $\forall m\in\mathbb{N}$,
\item For all $m\in\mathbb{N}$,
$$
|T|^m\le c_m^{T,\inv}(\mathbb{F}^\sigma T)\le|T'||T|^{m+1},
$$
where $T'=[T,T]$.
\end{enumerate}
\end{Thm}
\begin{proof}
For each sequence $s=(u_1,\ldots,u_m)\in T^m$, let
$$
B_s=\{u_{\pi(1)}\cdots u_{\pi(m)}\mid\pi\in\mathcal{S}_m\},
$$
and let (in $\mathbb{F}\langle X\rangle/\mathrm{Id}^{T,\inv}(\mathbb{F}^\sigma T)$)
$$
P_s^T(\mathbb{F}^\sigma T)=\mathrm{Span}\left\{x_{\pi(1)}^{(u_{\pi(1)})}\cdots x_{\pi(m)}^{(u_{\pi(m)})}\mid\pi\in\mathcal{S}_m\right\}.
$$
Since $x_{\pi(1)}^{(u_{\pi(1)})}\cdots x_{\pi(m)}^{(u_{\pi(m)})}=x_{\theta(1)}^{(u_{\theta(1)})}\cdots x_{\theta(m)}^{(u_{\theta(m)})}$ if and only if $u_{\pi(1)}\cdots u_{\pi(m)}=u_{\theta(1)}\cdots u_{\theta(m)}$ (see the proof of Theorem \ref{mainthmPI}), we see that
$$
\dim P_s^T(\mathbb{F}^\sigma T)=|B_s|.
$$
Hence, we may conclude that
\begin{equation}\label{eqn}
c_m^T(\mathbb{F}^\sigma T)=\sum_{s\in T^m}\dim P_s^T(\mathbb{F}^\sigma T)=\sum_{s\in T^m}|B_s|.
\end{equation}

Now, for each sequence $s\in T^m$, note that the elements of $B_s$ are constant modulo $[T,T]$ (since $T/T'$ is an abelian group). Thus, $B_s$ assumes at most $|T'|$ elements, that is, $1\le|B_s|\le|T'|$. From \eqref{eqn}, this gives us
$$
|T|^m\le c_m^{T}(\mathbb{F}^\sigma T)\le |T'||T|^m.
$$

Finally, consider a basis of $P^T_m(\mathbb{F}^\sigma T)$ consisting of monomials. Given $x_{i_1}^{(u_1)}\cdots x_{i_m}^{(u_m)}$, by the proof of Theorem \ref{mainthmPI}, the set
$$
\{\inv^{j_1}(x_{i_1}^{(u_1)})\cdots\inv^{j_m}(x_{i_m}^{(u_m)})\mid j_1,\ldots,j_m\in\{0,1\}\}
$$
contains at most $|T|$ distinct elements modulo $\mathrm{Id}_{T,\inv}(\mathbb{F}^\sigma T)$. Hence, $c_m^{T,\inv}(\mathbb{F}^\sigma T)\le|T|c_m^T(\mathbb{F}^\sigma T)$.
\end{proof}

As a consequence, we obtain the graded-involution exponent of $\mathbb{F}^\sigma T$:
\begin{Cor}
Let $T$ be a finite group, $\sigma\in Z^2(T,\mathbb{F}^\times)$ and $\inv$ a homogeneous involution on $\mathbb{F}^\sigma T$. Then $\mathrm{exp}^{T,\inv}(\mathbb{F}^\sigma T)=|T|$.\qed
\end{Cor}

Specializing Theorem \ref{mainthmPI} in the special case where $T$ is an abelian group, we obtain the exact codimension sequence:
\begin{Cor}
If $T$ is an abelian group and $\inv$ is a degree-preserving involution on $\mathbb{F}^\sigma T$, then for each $m\in\mathbb{N}$,
$$
c_m^{T,\inv}(\mathbb{F}^\sigma T)=c_m^{T}(\mathbb{F}^\sigma T)=|T|^m.
$$\qed
\end{Cor}

\section*{Acknowledgements}
The author would like to thank L.~A.~Mendon\c ca for the useful discussion, in special, for providing the part (2) of Theorem \ref{mainthmcodim}. Thanks to T.~de Mello for showing interest to this topic and reading and providing a feedback that improved this manuscript.


\begin{thebibliography}{22}
\bibitem{AljHN} E.~Aljadeff, D.~Haile, M.~Natapov, \emph{Graded identities of matrix algebras and the Universal graded algebra}, Transactions of the American Mathematical Society \textbf{362} (2010), 3125--3147.
\bibitem{BY} Y.~Bahturin, F.~Yasumura, \emph{Distinguishing simple algebras by means of polynomial identities}, S\~ao Paulo Journal of Mathematical Sciences \textbf{13}(1) (2019), 39--72.
\bibitem{BY2} Y.~Bahturin, F.~Yasumura, \emph{Graded polynomial identities as identities of Universal algebras}, Linear Algebra and its Applications \textbf{562} (2019), 1--14.
\bibitem{Mello} T.~de Mello, \emph{Homogeneous involutions on upper triangular matrices}, Archiv der Mathematik \textbf{118} (2022), 365--374.
\bibitem{Mello1} L.~F.~Fonseca, T.~de Mello, \emph{Degree-inverting involutions on matrix algebras}, Linear and Multilinear Algebra \textbf{66} (2018), 1104--1120.
\bibitem{Mello2} L.~F.~Fonseca, T.~de Mello, \emph{Graded polynomial identities for matrices with the transpose involution over an infinite field}, Communications in Algebra \textbf{46} (2018), 1630--1640.
\bibitem{VinKoVa2004} O.~Di Vincenzo, P.~Koshlukov, A.~Valenti, \emph{Gradings on the algebra of upper triangular matrices and their graded identities}, Journal of Algebra \textbf{275} (2004), 550--566.
\bibitem{EK2013} A.~Elduque, M.~Kochetov, \emph{Gradings on simple Lie algebras}. Providence: American Mathematical Society (2013).
\bibitem{FSY} L.~S.~Fonseca, E.~Santulo Jr., F.~Yasumura, \emph{Degree-inverting involution on full square and triangular matrices}, Linear and Multilinear Algebra \textbf{70} (2022), 1980--1994.
\bibitem{HN} D.~Haile, M.~Natapov, \emph{Graded polynomial identities for matrices with the transpose involution}, Journal of Algebra \textbf{464} (2016), 175--197.
\bibitem{Hazrat} R.~Hazrat, \emph{Graded Rings and Graded Grothendieck Groups}. Cambridge: Cambridge University Press (2016).
\bibitem{Razmyslov}  I.~P.~Razmyslov, \emph{Identities of algebras and their representations}, No. 138. American Mathematical Society (1994).
\bibitem{VZ2007} A.~Valenti, M.~Zaicev, \emph{Group gradings on upper triangular matrices}. Archiv der Mathematik \textbf{89} (2007), 33--40.
\end{thebibliography}
\end{document}